\documentclass{amsart}

\usepackage{calc}
\usepackage{ifthen}

\newcounter{boxsize}
\newcounter{tempcounter}

\newcommand\smbox{\put(0,0){\line(1,0){\value{boxsize}}}%
  \put(\value{boxsize},0){\line(0,1){\value{boxsize}}}%
  \put(0,0){\line(0,1){\value{boxsize}}}%
  \put(0,\value{boxsize}){\line(1,0){\value{boxsize}}}}
\newcommand\numbox[1]{\put(0,0)\smbox%
  \put(0,0){\makebox(\value{boxsize},\value{boxsize})[c]{%
      $\scriptscriptstyle#1$}}}
\newcommand\hdotbox{\setcounter{tempcounter}{\value{boxsize}*2}
  \multiput(-\value{boxsize},0)(0,\value{boxsize})2{%
    \line(1,0){\value{tempcounter}}}
  \put(\value{boxsize},0){\line(0,1){\value{boxsize}}}
  \put(-\value{boxsize},0){\makebox(\value{tempcounter},\value{boxsize})[c]{%
      $\scriptscriptstyle\cdots$}}}
\newcommand\vdotbox{\setcounter{tempcounter}{\value{boxsize}*2}
  \multiput(0,-\value{boxsize})(\value{boxsize},0)2{%
    \line(0,1){\value{tempcounter}}}
  \put(0,\value{boxsize}){\line(1,0){\value{boxsize}}}
  \put(0,-\value{boxsize}){\makebox(\value{boxsize},\value{tempcounter})[c]{%
      $\scriptscriptstyle\vdots$}}}
\newcommand\boxes[2]{\setcounter{tempcounter}{#1*\value{boxsize}/2}
  \multiput(0,-\value{tempcounter})(0,\value{boxsize}){#1}\smbox
  \ifthenelse{#2=1}{\put(0,-\value{tempcounter}){%
      \makebox(\value{boxsize},\value{boxsize})[c]{$\scriptscriptstyle1$}}}{}%
  \ifthenelse{#2=2}{\put(0,-\value{tempcounter}){%
      \makebox(\value{boxsize},\value{boxsize})[c]{$\scriptscriptstyle2$}}}{}%
  \ifthenelse{#2=3}{\put(0,-\value{tempcounter}){%
      \makebox(\value{boxsize},\value{boxsize})[c]{$\scriptscriptstyle2$}}%
    \setcounter{tempcounter}{\value{tempcounter}-\value{boxsize}}%
  \put(0,-\value{tempcounter}){\makebox(\value{boxsize},\value{boxsize})[c]{$\scriptscriptstyle1$}}}{}}

\newtheoremstyle{mytheorems}{9pt}{6pt}{\itshape}{0pt}{\sc}{.}{ }{}
\newtheoremstyle{myremarks}{6pt}{3pt}{\normalfont}{0pt}{\it}{.}{ }{}

\theoremstyle{mytheorems}
\newtheorem{theorem}{Theorem}[section]
\newtheorem{lemma}[theorem]{Lemma}

\theoremstyle{myremarks}

\newtheorem*{observation}{Key Observation}
\newtheorem*{remark}{Remark}

\renewcommand\qed{\phantom{m.} $\!\!\!\!\!\!\!\!$\nolinebreak\hfill\checkmark}

\DeclareMathOperator*\lto{\longrightarrow}
\DeclareMathOperator\im{{\rm im}}
\DeclareMathOperator\len{{\rm len}}

\begin{document}
\vglue1truecm
\centerline{\Large From Littlewood-Richardson Sequences}
\medskip\centerline{\Large to Subgroup Embeddings and Back}
\bigskip
\centerline{By}
\medskip
\centerline{\sc Markus Schmidmeier}

\bigskip\medskip

\centerline{\parbox{10cm}{\footnotesize 
    {\it Abstract.} Let $\alpha$, $\beta$, and $\gamma$ be partitions describing the
    isomorphism types of the finite abelian $p$-groups $A$, $B$, and $C$.
    From theorems by Green and Klein it is well-known that there is a short exact
    sequence $0\to A\to B\to C\to 0$ of abelian groups if and only if there is a 
    Littlewood-Richardson sequence of type $(\alpha,\beta,\gamma)$. 
    Starting from the observation that a sequence of partitions has the LR property
    if and only if every subsequence of length 2 does, we demonstrate how LR-sequences
    of length two correspond to embeddings of a $p^2$-bounded subgroup in a finite abelian
    $p$-group.  Using the known classification of all such embeddings we derive
    short proofs of the theorems by Green and Klein.
  }}
\renewcommand{\thefootnote}{}

\section{Littlewood-Richardson sequences.}\label{sectionLR}

Let $\Gamma=[\gamma^0,\ldots,\gamma^r]$ be an increasing sequence of partitions defining a tableau $T$.
We visualize $T$ by putting a number $h$ into each box in
the skew diagram $\gamma^h-\gamma^{h-1}$, for each $h=1,\ldots,r$.  
As usual, $\Gamma$ is a 
{\it Littlewood-Richardson sequence} or {\it LR-sequence} provided 
(LR1) each skew diagram $\gamma^h-\gamma^{h-1}$ is a horizontal strip and
(LR2) the word $w(T)$ obtained by reading off the numbers in $T$ row-wise, and within each row
from right to left, is a lattice permutation. 

\smallskip
Conditions (LR1) and (LR2) can be expressed in terms of the parts of the $\gamma^h$; 
but two different descriptions are obtained dependent on
whether those parts represent the rows (as in \cite[II.3]{macdonald}) 
or the columns of the corresponding Young diagram.  
In this manuscript, the parts of a partition will represent the columns, 
and then we have the  following characterization (as in \cite[2.1]{klein}):
\begin{enumerate}
\item[(LR1)] For each $h\geq 1$ and every $k$, 
  we have $0\leq\gamma^h_k-\gamma^{h-1}_k\leq1$.
\item[(LR2)] For each $h\geq 2$ and every $k$, the following inequality holds:
  $$\sum_{i\geq k} \big(\gamma^h_i-\gamma^{h-1}_i\big)
  \leq \sum_{i\geq k} \big(\gamma^{h-1}_i-\gamma^{h-2}_i\big)$$
\end{enumerate}

The {\it type} of the sequence $\Gamma=[\gamma^0,\ldots,\gamma^r]$ 
is the triple 
$(\alpha,\gamma^r,\gamma^0)$ of partitions where $\alpha$ is such that its 
conjugate $\alpha'$ has $h$-th part
$\alpha'_h=|\gamma^h-\gamma^{h-1}|$, counting the number of $h$'s in the 
tableau.

\begin{observation}\label{obs}
\sloppy
The sequence $\Gamma$ is an LR-sequence if and only if for each 
$h\geq 2$, the sequence $[\gamma^{h-2},\gamma^{h-1},\gamma^h]$ is an LR-sequence.
Here we put $\gamma^h=\gamma^r$ for $h\geq r$. 
Moreover, if $\Gamma$ has type $(\alpha,\gamma^r,\gamma^0)$ then each sequence
$[\gamma^{h-2},\gamma^{h-1},\gamma^h]$ has type 
$((\alpha_{h-1}',\alpha_h')',\gamma^h,\gamma^{h-2})$. 
\end{observation}

In this sense, arbitrary LR-sequences are generated by sequences of length 2.  
Such sequences are of the following form:

\begin{lemma}\label{r=2}
An increasing sequence $\Gamma=[\gamma^0,\gamma^1,\gamma^2]$ of partitions
is an LR-sequence if and only if two conditions are satisfied:
\begin{enumerate}
\item The set of columns in the tableau $T$ for $\Gamma$ is a totally ordered
  subset of the poset $\mathcal L$ with the ordering given by the horizontal position.
\begin{center}\hspace{-2em}
\begin{picture}(110,32)
\put(5,16){$\mathcal L:$}
\put(110,16){\boxes11}
\put(100,6){\boxes10}
\put(100,26){\boxes23}
\put(90,16){\boxes22}
\put(80,16){\boxes21}
\put(70,6){\boxes20}
\put(70,26){\boxes33}
\put(60,16){\boxes32}
\put(50,16){\boxes31}
\put(40,6){\boxes30}
\put(40,26){\boxes43}
\put(30,16){\boxes42}
\multiput(24.5,16)(30,0)3{\line(1,0){4}}
\multiput(34.5,19)(30,0)3{\line(1,1)4}
\multiput(34.5,13)(30,0)3{\line(1,-1)4}
\multiput(44.5,23)(30,0)3{\line(1,-1)4}
\multiput(44.5,9)(30,0)3{\line(1,1)4}
\put(15,16){$\cdots$}
\end{picture}
\end{center}
\item There is an injection
$\tau_{21}$ from the list $T_2$ of columns in $T$ of type 
\begin{picture}(12,3)\put(0,0){\numbox2}\put(3,0)\smbox\put(9,0){\hdotbox}
\end{picture}${\,}^t$
into the list $T_1$ of columns of type 
\begin{picture}(12,3)\put(0,0){\numbox1}\put(3,0)\smbox\put(9,0){\hdotbox}
\end{picture}${\,}^t$
which assigns to a column of length $\ell$ a column of length less than $\ell$. \qed
\end{enumerate}
\end{lemma}

\section{Finite length modules.}

Let $R$ be a commutative principal ideal domain and $p$ a generator of a maximal ideal. 
A $p$-{\it module} is a finite length $R$-module which is annihilated by some power of $p$.
There is a one-to-one correspondence between the set of partitions 
and the set of isomorphism classes of $p$-modules given by
$$\lambda=(\lambda_1,\ldots,\lambda_s)\quad\longmapsto\quad
M(\lambda)=\bigoplus_{i=1}^s R/(p^{\lambda_i})$$
(see \cite[(1.4)]{macdonald}). 
The partition corresponding to the $p$-module is its {\it type.}

\smallskip
Given a $p$-module $B$ and a submodule $A$ of $B$ of exponent $r$
then there is an increasing sequence of partitions
$[\gamma^0,\ldots,\gamma^r]$ given by the types of the factor modules
$B/p^hA$, where $0\leq h\leq r$. 

\smallskip
Embeddings $(A\subseteq B)$ where $p^2A=0$
have been classified (\cite[Theorem~7.5]{bhw},\cite[Corollary 5.4]{ps}):

\begin{theorem}\label{ps}
  Let $B$ be a $p$-module and $A$ a submodule of $B$ which is
  $p^2$-bounded.  Then the embedding $(A\subseteq B)$ has a direct
  sum decomposition, unique up to isomorphy and reordering,
  into finitely many indecomposable embeddings of type $P^\ell_m$ and $Q^{\ell s}_2$ defined as follows:
$$
    \begin{array}{r@{:\;\;}r@{\;}l@{\;\;\text{for}\;}l}
      P_m^\ell & (p^{\ell-m}) & \subseteq R/(p^\ell) 
        & \ell\in\mathbb N,\;0\leq m\leq \max\{\ell,2\},\\[1ex]
      Q^{\ell s}_2  & ((p^{\ell-2},p^{s-1})) & \subseteq R/(p^\ell)\oplus R/(p^s)
        & \ell,s\in\mathbb N,\;s<\ell-1.\qed
      \end{array}
  $$
\end{theorem}

\section{LR-sequences of length 2 and $p^2$-bounded
submodules}

The partition sequences of the indecomposable pairs $(A\subseteq B)$ with $p^2A=0$ are as follows.  
They are all LR-sequences.
\begin{center}
  \begin{picture}(110,24)
    \put(0,12){$P_0^\ell:$}
    \put(10,9){\begin{picture}(3,12)\put(0,0)\smbox \put(0,6)\vdotbox
        \put(3,3){$\left.\makebox(0,5){}\right\}{\!}_\ell$}
      \end{picture}}
    \put(30,12){$P_1^\ell:$}
    \put(40,8){\begin{picture}(3,15)\put(0,3){\boxes21}\put(0,9)\vdotbox
        \put(3,5){$\left.\makebox(0,6){}\right\}{\!}_\ell$}
      \end{picture}}
    \put(60,12){$P_2^\ell:$}
    \put(70,6){\begin{picture}(3,18)\put(0,4){\boxes33}\put(0,12)\vdotbox
        \put(3,6){$\left.\makebox(0,8){}\right\}{\!}_\ell$}
      \end{picture}}
    \put(87,12){$Q_2^{\ell s}:$}
    \put(105,0){\begin{picture}(6,24)
        \put(-6,11){${}_\ell\left\{\makebox(0,13){}\right.$}
        \put(0,3){\boxes22}\put(0,9)\vdotbox
        \put(0,15){\boxes20}\put(3,15){\boxes21}\put(0,21)\vdotbox
        \put(3,21)\vdotbox\put(6,17){$\left.\makebox(0,6){}\right\}{\!}_s$}
      \end{picture}}
  \end{picture}
\end{center}

When dealing with the direct sum of two pairs, the partition sequence 
of the sum is given by the union
  $\Gamma\cup\Delta=
  [\gamma^0\cup\delta^0,\gamma^1\cup\delta^1,\gamma^2\cup\delta^2]$ 
taken componentwise
  where $\Gamma=[\gamma^0,\gamma^1,\gamma^2]$ and 
  $\Delta=[\delta^0,\delta^1,\delta^2]$
  are the partition sequences of the two summands.
Note that the list of columns
  in the tableau for $\Gamma\cup\Delta$ is a reordering of the columns
  in the tableaux for $\Gamma$ and $\Delta$; however, each 
  (in $\mathcal L$ incomparable) pair of columns 
  arising from a sum $P_2^{\ell+1}\oplus P_0^\ell$ is replaced as follows:
  \begin{center}
    \begin{picture}(75,15)
      \put(0,7){Tableau for $P_2^{\ell+1}\oplus P_0^\ell$:}
      \put(47,4){\boxes33}\put(47,12)\vdotbox
      \put(52,7){$\cup$}
      \put(57,6){\boxes20}\put(57,12)\vdotbox
      \put(63,7){$=$}
      \put(69,4){\boxes32}\put(72,6){\boxes21}
      \put(69,12)\vdotbox \put(72,12)\vdotbox
    \end{picture}
  \end{center}
  
  Let $\Gamma=[\gamma^0,\gamma^1,\gamma^2]$ be the partition sequence of an embedding
$(A\subseteq B)$ with $p^2A=0$. We have seen
  that the columns of the corresponding tableau form a totally ordered subset
  of $\mathcal L$.  Columns of type 
  \begin{picture}(12,3)\put(0,0){\numbox2}\put(3,0)\smbox\put(9,0){\hdotbox}
    \end{picture}${\,}^t$
  arise for each summand of type $Q_2^{\ell s}$, where $s<\ell-1$,
  and for each pair of summands of type $P_2^{\ell+1}\oplus P_0^\ell$, as above.
  In each case there is a corresponding shorter column of type
  \begin{picture}(12,3)\put(0,0){\numbox1}\put(3,0)\smbox\put(9,0){\hdotbox}
    \end{picture}${\,}^t$.  
  The map $\tau_{21}$ given by this correspondence is a monomorphism, and
  Lemma~\ref{r=2} yields that  $\Gamma$ is an LR-sequence. The type of $\Gamma$
  is $(\alpha,\gamma^2,\gamma^0)$ where the partition $\alpha$ is given by
  $\alpha'_1=\len A/pA=|\gamma^2-\gamma^1|$, and
  $\alpha'_2=\len pA=|\gamma^1-\gamma^0|$.

\smallskip We have shown:

\begin{lemma}\label{lr1}
Let $A$, $B$, $C$ be $p$-modules of type $\alpha$, $\beta$ and $\gamma$, respectively,
such that $p^2A=0$ and such that there is a short exact sequence:
$$0\longrightarrow A\lto^\mu B\longrightarrow C\longrightarrow 0$$
For $h=0,1,2$, denote the type of $B/p^h\mu(A)$ by $\gamma^h$.
Then the partitions $[\gamma^0,\gamma^1,\gamma^2]$ form an LR-sequence
of type $(\alpha,\beta, \gamma)$.\qed
\end{lemma}

\smallskip
Conversely, every LR-sequence of length 2 gives rise to a short exact
sequence of $p$-modules:

\begin{lemma}\label{lr3}
Given are an LR-sequence $\Gamma=[\gamma^0,\gamma^1,\gamma^2]$ of length 2 and type
$(\alpha,\beta,\gamma)$, a $p$-module $B$ of type $\beta=\gamma^2$
and a semisimple submodule $U$ of $B$ such that $B/U$ has type $\gamma^1$.
Then there is a submodule $A$ of $B$ containing $U$ and satisfying the following
conditions:
\begin{itemize}
\item The type of $A$ is $\alpha$,
\item $U=pA$ and
\item the type of $B/A$ is $\gamma^0=\gamma$.
\end{itemize}
\end{lemma}

\begin{remark} A submodule $U$ as in the lemma can be obtained as follows. 
Let $\lambda=\gamma^2$ and define $\kappa$ by letting $\kappa_i=\gamma^2_i-\gamma^1_i$.
Then the direct sum $\bigoplus_i P_{\kappa_i}^{\lambda_i}$ defines an embedding 
$(U'\subseteq B')$ with $B\cong_\varphi B'$.
Since $[\gamma^1,\gamma^2]$ is an LR-sequence, $0\leq \kappa_i\leq 1$ follows, and 
hence $U'$ is semisimple. Put $U=\varphi^{-1}(U')$. 
\end{remark}

\begin{proof}
By Theorem~\ref{ps}, the embedding $(U\subseteq B)$ is isomorphic to a direct sum of embeddings
of type $P^\ell_0$ and $P^\ell_1$ as above. We may assume that $\varphi$ is the identity map on $B$
and write:
$$(*)\qquad (U\subseteq B) \;=\; \bigoplus_i P_{\kappa_i}^{\lambda_i}\quad
\text{where $\kappa$ is as above and $\lambda=\gamma^2$.}$$
Since $\Gamma$ is an LR-sequence, say definining the tableau $T$,
Lemma~\ref{r=2} gives us an injective map 
$\tau_{21}:T_2\to T_1$ between the two lists of columns containing only a 2
and containing only a 1, respectively.
If the $i$-th column of the tableau (which is defined by the numbers
$(\gamma_i^0,\gamma_i^1,\gamma_i^2)$) does not occur among the columns 
in $T_2\cup\im\tau_{21}$, let $(A_i\subseteq B_i)$ denote the $i$-th summand 
$P_{\kappa_i}^{\lambda_i}$ in $(*)$.
Otherwise $\tau_{21}$ defines a pair $(i,j)\in T_2\cup T_1$ where $j=\tau_{21}(i)$
such that the $i$-th column has length $\ell$ and type 
\begin{picture}(12,3)\put(0,0){\numbox2}\put(3,0)\smbox\put(9,0){\hdotbox}
\end{picture}${\,}^t$
and the $j$-th column has length $s<\ell$ and type 
\begin{picture}(12,3)\put(0,0){\numbox1}\put(3,0)\smbox\put(9,0){\hdotbox}
\end{picture}${\,}^t$.
Let $B_i$ be the sum of the $i$-th and the $j$-th summand in $(*)$, then
choose $A_i\subseteq B_i$ such that 
$$(A_i\subseteq B_i)\quad\cong\quad\left\{\begin{array}{cl}
    P^\ell_2\oplus P^s_0 & \text{if}\;s=\ell-1\\
    Q^{\ell s}_2 & \text{if}\;s<\ell-1\end{array}\right..$$
Putting $A=\bigoplus_{i\in\{1,\ldots,n\}\backslash\im\tau_{21}}
    A_i$ yields a submodule of $B$ of type $\alpha$ such that $U=pA$.
The type of $B/A$ is computed from the types of the summands, as above
Lemma~\ref{lr1}.
\end{proof}

\section{The Theorems by Green and Klein}

Following \cite{klein}, we can now establish the correspondence between 
LR-sequences and short exact sequences of $p$-modules.

\begin{theorem}[Green]
Suppose $A$, $B$, and $C$ are $p$-modules of type 
$\alpha$, $\beta$, and $\gamma$ such that there is a short exact
sequence
$$0\lto A\lto^\mu B\lto C\lto 0.$$
Then the sequence of partitions $\Gamma=[\gamma^0,\ldots,\gamma^r]$
where $r=\alpha_1$ is the exponent of $A$ and $\gamma^h$ is the type of $B/p^h\mu(A)$ 
for $0\leq h\leq r$, is an LR-sequence of type $(\alpha,\beta,\gamma)$.
\end{theorem}

\begin{proof}
\sloppy
We may assume that $\mu$ is an inclusion, then $A$ is a submodule of $B$.
For each $2\leq h\leq r$, we have the short exact sequence
$$0\lto p^{h-2}A/p^hA\lto^\mu B/p^hA \lto B/p^{h-2}A\lto 0.$$
The modules $B/p^hA$, $(B/p^hA)/(p^{h-1}A/p^hA)\cong B/p^{h-1}A$, and $B/p^{h-2}A$
have type $\gamma^h$, $\gamma^{h-1}$, and $\gamma^{h-2}$, respectively, and 
hence, by Lemma~\ref{lr1}, the sequence  $[\gamma^{h-2},\gamma^{h-1},\gamma^h]$
is an LR-sequence of type $\big((\alpha_{h-1}',\alpha_h')',\gamma^h,\gamma^{h-2}\big)$
where $\alpha_h'=|\gamma^h-\gamma^{h-1}|$ and $\alpha_{h-1}'=|\gamma^{h-1}-\gamma^{h-2}|$.
It follows from the Key Observation in Section~\ref{sectionLR} that $\Gamma$ 
is an LR-sequence of type $(\alpha,\gamma^r,\gamma^0)$. 
\end{proof}

\begin{theorem}[Klein]\label{klein}
If $\Gamma=[\gamma^0,\ldots,\gamma^r]$ is an LR-sequence of type $(\alpha,\beta,\gamma)$,
and if $A$, $B$, and $C$ are $p$-modules of type $\alpha$, $\beta$, and $\gamma$,
respectively, 
then there is a short exact sequence
$$0\lto A \lto^\mu B\lto C\lto 0$$
such that $\gamma^h$ is the type of $B/p^h\mu(A)$ for each $0\leq h\leq r$.
\end{theorem}

\begin{proof}
We may assume that $r\geq 2$.
First consider the LR-sequence $[\gamma^{r-2},\gamma^{r-1},\gamma^r]$.
By Lemma~\ref{lr3} and the following remark there is a $p$-module $B$ 
and a submodule $U_{r-2}\subseteq B$ such that $B$, $B/pU_{r-2}$ and $B/U_{r-2}$
have type $\gamma^r$, $\gamma^{r-1}$, and $\gamma^{r-2}$, respectively. 
We put $U_r=0$, $U_{r-1}=pU_{r-2}$ and define for $h=r-3,r-4,\ldots,0$
successively submodules $U_h$ of $B$ such that the conditions
$$p^sU_h=U_{h+s}\quad\text{and}\quad B/U_h\;\text{has type}\;\gamma^h$$
hold for all $0\leq h\leq h+s\leq r$.
Suppose $U_{h+1}$ has been constructed.  Consider the module $B'=B/pU_{h+1}$
and the semisimple submodule $U'=U_{h+1}/pU_{h+1}$.  
By Lemma~\ref{lr3}, the LR-sequence $[\gamma^{h},\gamma^{h+1},\gamma^{h+2}]$ 
yields a submodule $A'$ of $B'$ such that $pA'=U'$ and $B'$, $B/U'$, $B/A'$
have type $\gamma^{h+2}$, $\gamma^{h+1}$, and $\gamma^{h}$, respectively.
Let $U_h$ be the inverse image of $A'$ under the
canonical map $B\to B'$.

\smallskip
This process yields a submodule $A=U_r$ of $B$ with the property that the type
of $B/p^hA=B/U_h$ is $\gamma^h$ for each $0\leq h\leq r$. 
\end{proof}

\bigskip\bigskip
Address of the author:\hfill
\begin{minipage}[t]{7cm}
  Mathematical Sciences\\
  Florida Atlantic University\\
  Boca Raton, Florida 33431-0991\\
  United States of America\\[1ex]
  e-mail: {\tt markus@math.fau.edu}
\end{minipage}

\end{document}